\documentclass[11pt]{article}
\usepackage{amsmath, amssymb, amsthm, amsfonts}
\usepackage{authblk}
\usepackage{fixltx2e}

\usepackage{amsmath, amssymb, amsthm, amsfonts}
 \newtheorem{theorem}{Theorem}
 \newtheorem{lemma}[theorem]{Lemma}
  \newtheorem{proposition}[theorem]{Proposition}
 \newtheorem{corollary}[theorem]{Corollary}

\usepackage{amssymb}
\usepackage{color}
\usepackage{amsmath}
\usepackage{amsthm}
\usepackage{comment}

\newcommand{\mbb}{\mathbb}
\newcommand{\N}{\mathbb N}
\newcommand{\Z}{\mathbb Z}
\newcommand{\R}{\mathbb R}
\newcommand{\C}{\mathbb C}

\let\ol=\overline

\newcommand{\T}{\overline{T}}

\newcommand{\ms}{\mathfrak{S}}
\newcommand{\Id}{\rm Id}

\newcommand{\D}{\mathcal D}

\theoremstyle{remark}
\newtheorem{remark}[theorem]{Remark}

\theoremstyle{definition}

\title{Dynamics of Certain Distal Actions on Spheres}
\author{Riddhi Shah}
\author{Alok Kumar Yadav}
\affil{Jawaharlal Nehru University, New Delhi, India}

\begin{document}

\maketitle
\begin{abstract}
Consider the action of $SL(n+1,\R)$ on $\mbb{S}^n$ arising as the quotient of the linear action on $\R^{n+1}\setminus\{0\}$. We show that for a semigroup $\mathfrak{S}$ of $SL(n+1,\R)$, the following are equivalent:  $(1)$ $\mathfrak{S}$ acts distally on the unit sphere 
$\mbb{S}^n$. $(2)$ the closure of $\mathfrak{S}$ is a compact group. We also show that if $\mathfrak{S}$ is closed, the above conditions are equivalent to the condition that 
every cyclic subsemigroup of $\mathfrak{S}$ acts distally on $\mbb{S}^n$. On the unit circle $\mbb{S}^1$,  we consider the `affine' actions corresponding to maps in  
$GL(2,\R)$ and discuss the conditions for the existence of fixed points and periodic points, which in turn imply that these maps are not distal. \end{abstract}

\smallskip



\section{Introduction}
Let $X$ be a (Hausdorff) topological space. A semigroup $\mathfrak{S}$ of homeomorphisms of $X$ is said to act {\it distally} on $X$ if for any pair of distinct elements $x, y\in X$, 
the closure of $\{(T(x),T(y))\mid T\in\mathfrak{S}\}$ does not intersect the diagonal $\{(d,d)\mid d\in X \}$; (equivalently we say that the $\ms$-action on $X$ is distal). 
Let $T:{X}\to{X}$ be a homeomorphism. The map $T$ is said to be {\it distal} if the group 
$\{T^n\}_{n\in\Z}$ acts distally on $X$. If $X$ is compact, then $T$ is distal if and only if the semigroup $\{T^n\}_{n\in\N}$ acts distally (cf.\ Berglund et al.\ \cite{BJM}). 

The notion of distality was introduced by Hilbert (cf.\ Ellis \cite{E4}, Moore \cite{M9}) and studied by many in different contexts (see Abels \cite{A1,A2}, 
Furstenberg \cite{F6}, Raja-Shah \cite{RaSh10,RaSh11} and Shah \cite{Sh12}, and references cited therein). Note that a homeomorphism $T$ of a 
topological space is distal if and only if $T^n$ is so, for any $n\in\N$.  

For $x\in \R^n$, let $\|x\|$ denote the usual norm, which is the Euclidean distance between $x$ and the origin. For $T\in GL(n+1,\R)$ and $x\in\R^{n+1}\setminus\{0\}$, let 
$\overline{T}:\mbb{S}^n \to \mbb{S}^n$ be defined as $\overline{T}(x)= T(x)/ \|T(x)\|$, $x\in \mbb{S}^n$. 
Let $\D=\{r\Id\in GL(n+1,\R)\mid r\in\R\setminus\{0\}\}$. Then $\D$ is the centre of $GL(n+1,\R)$. 

In Section 2, we shall prove that for a semigroup $\mathfrak{S}$ of $GL(n+1,\R)$, distality of the action of $\mathfrak{S}$ on $\mbb{S}^n$ and the condition that the closure of 
$\mathfrak{S}\D/\D$ in $GL(n+1,\R)/\D$ is a compact group are equivalent (see Theorem~\ref{a}). We also show that if $\mathfrak{S}$ is a closed semigroup of $SL(n+1,\R)$ 
(consisting of matrices of determinant 1), 
then the above is equivalent to the condition that every cyclic subsemigroup of $\mathfrak{S}$ acts distally on $\mbb{S}^n$ (see Corollary~\ref{c}). In Section 3, for $T$ as above 
 and $a\in\R^{n+1}\setminus\{0\}$ with $\|T^{-1}(a)\|<1$, we discuss the `affine' action of $\overline{T}_a$ on $\mbb{S}^n$ defined by 
 $\displaystyle{\overline{T}_a(x)=(a+T(x))/ \|a+T(x)\|}$, $x\in\mbb{S}^n$. Here, $\overline{T}_a$ is well defined and a homeomorphism (see Lemma~\ref{e}). For $T\in GL(2,\R)$, 
 we study distality of `affine' maps on $\mbb{S}^1$. In particular, we show that if $T$ has at least one positive real eigenvalue then 
$\overline{T}_a$ has a fixed point and hence it is not distal. If $T$ has complex eigenvalues belonging to a specific region in $\mbb{S}^1$, then there exists $a$ such that 
$\overline{T}_a$ on $\mbb{S}^1$ has fixed points. We also explore some conditions to ensure the existence of at least one nonzero $a$ such that $\T_a$ on $\mbb{S}^n$ is not 
distal.

For a $T\in GL(n,\R)$, let $C(T)=\{v\in \R^n\mid T^m(v)\to 0\mbox{ as }m\to\infty\}$. Note that 
$T$ is distal on $\R^n$ if and only if $C(T)$ and $C(T^{-1})$ are trivial. We will use the notion of contraction groups below.


\section{Dynamics of the semigroup action on $\mbb{S}^n$}
In this section we consider a semigroup of $GL(n+1,\R)$ and study the dynamical properties of its canonical actions on $\mbb{S}^n$. 
Note that if $T\in GL(n+1, \R)$ is distal, it does not imply that $\overline{T}$ is distal. For example, if we consider $T=(a_{ij})$, a $2\times 2$ matrix with 
entries $a_{11}=1,a_{12}=1, a_{21}=0$ and $a_{22}=1$, then it is easy to check that $T$ is distal on $\R^2$ but $\overline{T}$ on $\mbb{S}^1$ is not distal. 
The following theorem characterises distal actions of semigroups on $\mbb{S}^n$. Here, $\mathfrak{S}\subset GL(n+1,\R)$ acts on $\mbb{S}^n$ through the 
continuous (group) action of $GL(n+1,\R)$ on $\mbb{S}^n$ which is defined as follows, for $T\in GL(n+1,\R)$. 
\begin{center}
$GL(n+1,\R)\times\mbb{S}^n\to\mbb{S}^n$, $(T, x)\mapsto \overline{T}(x)$, where $x\in\mbb{S}^n$,
\end{center}

\begin{theorem}\label{a}
Let $\mathfrak{S}\subset GL(n+1,\R)$ be a semigroup. Then the following are equivalent:
\begin{enumerate}
\item[$(a)$] $\mathfrak{S}$ acts distally on $\mbb{S}^n$.
\item[$(b)$] The closure of $\ms\D/\D$ in $GL(n+1,\R)/\D$ is a compact group, where $\D$ is the centre of $GL(n+1,\R)$.
\item[($c)$] For the semigroup $\mathfrak{S}^\prime=\{\alpha_TT\mid T\in GL(n+1,\R) \mbox{ and }\alpha_T=|\det T|^{-1/n}\}$, the closure of $\mathfrak{S}^\prime$ is a 
compact group.
\end{enumerate}
\end{theorem}
 
We first prove another result which will be required in the proof of the above theorem.

In the following proposition, for $T\in GL(n,\C)$, where $\C$ is the field of complex numbers, the condition that $\overline{\{T^m\}}_{m\in\N}$ is compact, is equivalent to 
the condition that  $\overline{\{T^m\}}_{m\in\N}$ is a compact group (cf.\ \cite{HeMo7}). It is also equivalent to the condition that $T$ is semisimple and its eigenvalues 
are of absolute value one; (this fact is well-known). Therefore the following proposition shows that Theorem 1.1 of \cite{FM5} holds for semigroups. 

\begin{proposition}\label{b}
If $\mathfrak{S}$ is a closed semigroup in $GL(n,\C)$ such that for every element $T\in\mathfrak{S}$, the closure of $\{T^m\}_{n\in\N}$ in $GL(n,\C)$ is compact. Then 
$\mathfrak{S}$ is a compact group.
\end{proposition}

\begin{proof} Let $H_T=\overline{\{T^m\}}_{m\in\N}$. Since $H_T$ is compact, it is a group (see Sec.\ 1 in Ch.\ $A$ of \cite{HeMo7}). Therefore every 
element of $\mathfrak{S}$ is invertible and hence $\mathfrak{S}$ itself is a group. Moreover, as $H_T$ is a compact group, $T$ is semisimple and eigenvalues  of $T$ 
are of absolute value one. Hence by Theorem 1.1 of \cite{FM5}, $\mathfrak{S}$ is contained in a conjugate of unitary group $U_n(\C)$. In 
particular, $\mathfrak{S}$ is a compact group.
\end{proof}

\begin{remark}
In Proposition~\ref{b} we have considered a closed semigroup because there exists a subgroup of $GL(n,\C)$ with non-compact closure such that every element of the subgroup 
generates a relatively compact group \cite{BH}.
\end{remark}

\smallskip
\noindent{Proof of Theorem~\ref{a}:} $(a)\Rightarrow (c):$ Suppose the $\mathfrak{S}$-action on $\mbb{S}^n$ is distal. Since the action of $\mathfrak{S}^\prime$ on $\mbb{S}^n$ is 
same as that of $\mathfrak{S}$, we have that the action of $\mathfrak{S}^\prime$ on $\mbb{S}^n$ is distal. Moreover, as the closure of $\mathfrak{S}^\prime$ is also a 
semigroup whose elements have determinant $\pm1$, and it acts distally on $\mbb{S}^n$, we may assume that $\mathfrak{S}^\prime$ is closed. We first show that for 
every $T\in\mathfrak{S}^\prime$, $\{T^n\}_{n\in\N}$ is relatively compact. Let $T\in\mathfrak{S}^\prime$ be fixed. As $\det T=\pm 1$, at least one of the 
following holds: (i) all the eigenvalues of $T$ are of absolute value one, (ii) at least one eigenvalue of $T$ has absolute value less than one and at least one eigenvalue of $T$ 
has absolute value greater than one.

If possible, suppose $\{T^m\}_{m\in \N}$ is not relatively compact in $\mathfrak{S}'$. Then there exists $\{m_k\}\subset \N$ such that $\{T^{m_k}\}$ is divergent, i.e.\ it has no 
convergent subsequence. Moreover, we show that there exist a subsequence of $\{m_k\}$, which we denote by $\{m_k\}$ again, and a nonzero vector $v_0$ such that 
$\{T^{m_k}(v_0)\}$ converges.

 Suppose (i) holds. If $1$ or $-1$ is an eigenvalue of $T$, then there exists a nonzero eigenvector $v$ such that $T(v)=v$ or $T(v)=-v$. If all the eigenvalues of $T$ are 
 complex and of absolute value one, then there exists a two dimensional subspace $W^\prime$ such that $T|_{W^\prime}$ being conjugate to a rotation map, generates a 
relatively compact group in $GL(W^\prime)$. Hence, for every $v_0\in W^\prime\setminus\{0\}$, $\{T^{m_k}(v_0)\}$ is relatively compact and has a subsequence which 
converges.
 
Now suppose (ii) holds. Then $T$ has at least one eigenvalue of absolute value less than one. Therefore, the contraction group 
$C(T)$ is nontrivial, and we can choose $v_0$ as any nonzero vector in $C(T)$.

By Lemma 2.1 of \cite{DaRa3}, there exist a subspace $W$ of $\R^{n+1}$ and a subsequence $\{l_k\}$ of $\{m_k\}$ such that $\{T^{l_k}(v)\}$ converges for every 
$v\in W$ and $\|T^{l_k}(v)\|\to\infty$, whenever $v\not\in W$. Here $W\neq\{0\}$ as $v_0\in W$.

Now we show that $W=\R^{n+1}$. If possible, suppose $W\neq \R^{n+1}$. Then there exists $u\in \R^{n+1}\setminus W$ such that $\|T^{l_k}(u)\|\to\infty$. Since 
$\mbb{S}^n$ is compact, 
passing to a subsequence if necessary, we have $\ol{T}^{l_k}(u)=T^{l_k}(u)/ \| T^{l_k}(u)\|\to a$, for some $a\in \mbb{S}^n$. Let $v_0\in W$ be as above. As 
$u\notin W, u+v_0\notin W$ 
and therefore $T^{l_k}(u+v_0)\to\infty$. As $\{T^{l_k}(v_0)\}$ is bounded, we get that $\ol{T}^{l_k}(u+v_0)\to a$. Here, $\overline{u}\neq\overline{u+v_0}$ as $v_0\in W$ 
and $u\notin W$. 
This is a contradiction as $\mathfrak{S}$ acts distally on $\mbb{S}^n$. Hence $W=\R^{n+1}$, and therefore $\{T^{l_k}\}$ is bounded. As $\{l_k\}$ is a subsequence of 
$\{m_k\}$, we arrive 
at a contradiction to our earlier assumption that $\{T^{m_k}\}$ is divergent. Hence $\{T^m\}_{m\in{\N}}$ is relatively compact and all its limit points belong to 
$\mathfrak{S}^\prime$. Therefore, by Proposition~\ref{b}, $\mathfrak{S}^\prime$ is a compact group. 

$(c)\Rightarrow (b)$ as $\ol{\ms\D}/\D=\ol{\ms'}\D/\D$ which is a compact group. 
It is easy to see that $(b)\Rightarrow(a)$ as on $\mbb{S}^n$, the action of $\ms$ is same as the action of $\ms\D/\D$ whose closure is a compact group which acts distally. \qed

\begin{remark}
From Theorem~\ref{a}, it follows that for a semigroup $\mathfrak{S}\subset GL(n+1,\R)$ whose all elements have determinant $1$ or $-1$, then the 
distality of the  $\mathfrak{S}$-action on $\mbb{S}^n$ implies the distality of the $\mathfrak{S}$-action on $\R^{n+1}$. In the latter part of the proof of the theorem
instead of \cite{FM5}, one can also use the results about the structure of distal linear groups from \cite{A1} and give a different argument. 
\end{remark}

There are examples of actions of semigroups on compact spaces which are not distal but every cyclic subsemigroup acts distally, (see \cite{JR}). However, the latter does not 
happen in the case of closed semigroups of $SL(n+1,\R)$ for the action on $\mbb{S}^n$. 

\begin{corollary} \label{c}
For a closed semigroup $\mathfrak{S}$ of $SL(n+1,\R)$, the following holds: $\mathfrak{S}$ acts distally on 
$\mbb{S}^n$ if and only if every cyclic semigroup of $\mathfrak{S}$ acts distally on $\mbb{S}^n$.
\end{corollary}

\begin{proof} The``only if" statement is obvious. Now suppose every cyclic semigroup of $\mathfrak{S}$ acts distally on $\mbb{S}^n$. Let $T\in\mathfrak{S}$. 
By Theorem~\ref{a}, $\overline{\{T^m\}}_{m\in\N}$ is a compact group. Now the assertion follows from Proposition~\ref{b}.
\end{proof}


\section{Dynamics of `affine' maps on $\mbb{S}^n$}

Consider the affine action on $\R^{n+1}$, $T_a(x)=a+T(x)$, where $T\in GL(n+1,\R)$, and $a\in \R^{n+1}$. In this section, we first consider the corresponding 
`affine' map $\overline{T}_a$ on $\mbb{S}^n$ which is defined for any nonzero $a$ satisfying $\|T^{-1}(a)\|\ne 1$ as follows: $\overline{T}_a(x)=T_a(x)/\|T_a(x)\|, x\in \mbb{S}^n$. 
(For $a=0$, $\overline{T}_a=\overline{T}$, which is studied in Section 2). Observe that $T_a(x)=0$ for some $x\in \mbb{S}^n$ if and only if $T^{-1}(a)$ has norm 1. Therefore, $\ol{T}_a$ is well 
defined if $\|T^{-1}(a)\|\ne 1$. The map $\T_a$ is a homeomorphism for any nonzero $a$ satisfying 
$\|T^{-1}(a)\|<1$ (see Lemma~\ref{e}). In this section, we study the dynamics of such homeomorphsims $\overline{T}_a$.
 
Note that any nontrivial homeomorphism $S$ of $\mbb{S}^1$ with a fixed or a periodic point is not distal unless some power of $S$ is an identity map. In fact if 
$S$ has a fixed point or a periodic point of order $2$, then either $S^2=\Id$, or there exist $x,y\in\mbb{S}^1$, $x\ne y$, such that $S^{2n}(x)\to z$ and 
$S^{2n}(y)\to z$ for some fixed point $z$ of $S^2$.  This can be seen 
through an identification of $\mbb{S}^1$ to $[0,1]$ and getting an increasing homeomorphism of $[0,1]$ equivariant to $S^2$, as the latter is orientation preserving. 
These facts are well-known, we refer the reader to \cite{P} and \cite{BS} for more details. We will discuss the existence of fixed points or periodic points of order 2 for $\overline{T}_a$ on 
$\mbb{S}^1$, $T \in GL(2,\R)$, under certain conditions on the eigenvalues of $T$ and the norm of $T$. Note that for such a homeomorphism $\T_a$ on $\mbb{S}^1$, $\T_a^2$ is 
nontrivial; if it were trivial, then for every $x\in \mbb{S}^1$, either $x$ or $-x$ would belong to the positive cone generated by $a$ and $T(a)$ in $\R^2$, which would lead to a contradiction.  Therefore, 
if $\T_a$ has a fixed point or a periodic point of oder 2, then $\T_a$ is not distal. 

\begin{lemma}\label{e}
Let $T\in GL(n+1, \R)$ and let $a\in\R^{n+1}\setminus\{0\}$ be such that $\|T^{-1}(a)\|\ne 1$. The map $\overline{T}_a$ on $\mbb{S}^n$ is a homeomorphism if and only 
if $\|T^{-1}(a)\|<1$.
\end{lemma}

\begin{proof}
Suppose $\|T^{-1}(a)\|<1$. From the definition, it is clear that $\overline{T}_a$ is continuous. It is enough to show that $\overline{T}_a$ is a bijection since any continuous bijection 
on a compact Hausdorff space is a homeomorphism.
Suppose $x,y\in \mbb{S}^n$ such that $\overline{T}_a(x)=\overline{T}_a(y)$. Then we have $\left(a+T(x)\right)/\|a+T(x)\|=\left(a+T(y)\right)/\|a+T(y)\|$ or $(1-\beta)T^{-1}(a)=\beta y-x$, 
where $\beta=\|a+T(x)\|/\|a+T(y)\|$. If $\beta\neq 1$, then we get that $\|T^{-1}(a)\|\geq 1$, a contradiction. Hence $\beta=1$, 
and $x=y$. Therefore, $\overline{T}_a$ is injective.

Let $y\in \mbb{S}^n$ be fixed. Let $\Psi: \R^+\to\R^+$ be defined as follows: $\Psi(t)=\|tT^{-1}(y)-T^{-1}(a)\|$, $t\in\R^+$. Clearly, $\Psi$ is a continuous map, and hence the 
image of $\Psi$ is connected. We have $\Psi(0)=\|T^{-1}(a)\|<1$ and $\Psi(t)\to\infty$ as $t\to\infty$. Therefore there exists a $t_0\in\R^+$, $t_0\neq{0}$ such that $\Psi(t_0)=1$. Let 
$x=t_0T^{-1}(y)-T^{-1}(a)$. Then $x\in\mbb{S}^n$ and $\overline{T}_a(x)=y$. Hence $\overline{T}_a$ is surjective. 

Conversely, if $\|T^{-1}(a)\|>1$, then $\T_a(x)=\T_a(-x)=a/\|a\|$, for 
$x=T^{-1}(a)/\|T^{-1}(a)\|\in \mbb{S}^n$, i.e.\ $\overline{T}_a$ is not injective. 
\end{proof}

Observe that, $\R^2$ is isomorphic to the field $\C$ of complex numbers and $\mbb{S}^1$ is a group under multiplication. For $x\in\R^2\setminus\{0\}$, we take $x^{-1}$ as the 
inverse of $x$ in $\C$.

\begin{theorem}\label{m}
Let $T\in GL(2,\R)$ and let $a\in\R^2\setminus\{0\}$ be such that $\|T^{-1}(a)\|<1$. Then the following hold:
\begin{enumerate}
\item[$(1)$]  If an eigenvalue of $T$ is real and positive, then $\overline{T}_a$ has a fixed point. 
\item[$(2)$] If the eigenvalues of $T$ are complex of the form $r=t(\cos\theta\pm i\sin\theta)$, ($t>0$). Then $T=tABA^{-1}$ for some $A$ in $GL(2,\R)$ and $B$ is a 
rotation by the angle $\theta$. Suppose $\cos{\theta}>0$ and $|\sin\theta|\leq\|T^{-1}(a)\|/(\|A\|\|A^{-1}\|)$. Then $\overline{T}_a$ has a fixed point. 
\end{enumerate}
\end{theorem}

\begin{remark} If $T$ has complex eigenvalues, then we may assume in (2) above that $\det A=\pm 1$ and $A$ is 
unique up to isometry. For if $ABA^{-1}=CBC^{-1}$ where $B$ is a rotation and $B\ne\pm \Id$, then $C^{-1}A$ commutes with 
$B$ and hence it is a rotation. Therefore, $\|A\|\|A^{-1}\|$ is uniquely defined for any such $T$. If $t^{-1}T$ is a rotation by an angle $\theta$, then we may take 
$A=\Id$ and if $\cos\theta>0$ and $|\sin\theta|\leq\|a\|<1$, then $\ol{T}_a$ has a fixed point. (It is also easy to check that the converse also holds in this case). 
\end{remark}

\noindent{Proof of Theorem~\ref{m}:}
Note that $\overline{T}_a=(\overline{\beta T})_{(\beta a)}$ for all $\beta>0$. Without loss of any generality, we can assume that $T\in GL(2,\R)$ such that 
$\det T=\pm 1$. Observe that $\T_a$ has a fixed point if there exists $\gamma>0$ such that $\gamma\Id-T$ is invertible and $x_\gamma=(\gamma\Id-T)^{-1}(a)$ has norm 1. 
However, such a $\gamma$ may not exist. Hence, we deal with some of the special cases below separately. 

\smallskip
\noindent$(1)$ Suppose $T$ has a positive real eigenvalue. There exists $A\in GL(2,\R)$ such that $T=ABA^{-1}$ where 
$B= \begin{bmatrix}
t & 0\\
0 & s
\end{bmatrix}$ or 
$\begin{bmatrix}
1 & 1\\
0 & 1
\end{bmatrix}$ for $t>0$ and $ts=\pm 1$. Let $a\in\R^2\setminus\{0\}$ be fixed such that $\|T^{-1}(a)\|<1$. Let $A^{-1}(a)=(a_1,a_2)$.

Now consider $B=\begin{bmatrix}
t & 0\\
0 & s
\end{bmatrix}$.
Note that if $a_2=0$ then $\overline{a}=a/\|a\|$ is a fixed point of $\overline{T}_a$.
If $s=1$ then $T=\Id$, and $\overline{a}$ is a fixed point of $\overline{T}_a$ for $a\in\R^2$ as above. Now let $s\neq 1$.
Suppose $a_1=0$. If $s>0$ then $\overline{a}$ is a fixed point of $\overline{T}_a$. Now suppose $s<0$. Then $ts=-1$ and $t-s>0$.  
Here, $\|T^{-1}(a)\|=\|AB^{-1}A^{-1}(a)\|=\|A(s^{-1}A^{-1}(a))\|=\|A(0,s^{-1}a_2)\|=\|ta\|<1$ (which is given). Therefore, $\|A(0,a_2/(t-s))\|=\|(t-s)^{-1}a\|=(t^2+1)^{-1}\|ta\|<1$. Moreover, as 
$A$ is an invertible linear map, $\|A(x,a_2/(t-s))\|\to\infty$ as $|x|\to\infty$. Therefore there exists a real number $x_0\ne 0$ such that 
$\|A(x_0, a_2/(t-s))\|=1$. It is easy to see that $x=A(x_0,a_2/(t-s))$ is a fixed point of $\overline{T}_a$.

Let $a_1$ and $a_2$ be nonzero. Consider $f:\R\setminus\{t,s\}\to\R^+$ defined by $f(\gamma)=\|A(a_1/(\gamma-t),a_2/(\gamma-s))\|$. As $\gamma\to 0$, 
$f(\gamma)\to \|T^{-1}(a)\|<1$, and as $\gamma\to t_0$, $f(\gamma)\to\infty$, where $t_0=\min\{t,s\}$ if $s>0$, and $t_0=t$ if $s<0$. Therefore 
there exists a $\gamma_0\in\, ]0,t_0[$ such that $\|A(a_1/(\gamma_0-t),a_2/(\gamma_0-s))\|=1$. It is easy to check that 
$A(a_1/(\gamma_0-t),a_2/(\gamma_0-s))$ is a fixed point of $\overline{T}_a$.

Let $B= \begin{bmatrix}
1 & 1\\
0 & 1
\end{bmatrix}$. If $a_2=0$, then $\ol{a}$ is a fixed point for $\ol{T}_a$. As $\|T^{-1}(a)\|<1$, arguing as above we can find a $\gamma_1\in]0,1[$ and show 
that $A(a_1/(\gamma_1-1)+a_2/(\gamma_1-1)^2,a_2/(\gamma_1-1))$ has norm one and it is a fixed point of $\overline{T}_a$. 

\smallskip
\noindent $(2)$ As $T$ has complex eigenvalues $t(\cos\theta\pm i\sin\theta)$, ($t>0$), we have that $\det T>0$. Hence we may assume that 
$\det T =1$ and $T=ABA^{-1}$, where $B$ is a rotation by the angle $\theta$. Let $r_1=\cos{\theta}>0$, $r_2=\sin{\theta}$ and $A^{-1}(a)=(a_1,a_2)$. 
Let $g:\R^+\to\R^+$ be defined as $g(\gamma)=\|A(\gamma\Id-B)^{-1}A^{-1}(a)\|$. Then 
$g(0)=\|T^{-1}(a)\|<1$. As $B$ is an isometry, we have that $\|A^{-1}(a)\|=\|A^{-1}T^{-1}(a)\|$. Here,
$g(r_1)\geq |r_2|^{-1}\|A^{-1}(a)\|/\|A^{-1}\|\geq |r_2|^{-1}\|T^{-1}(a)\|/[\|A\|\|A^{-1}\|]\geq 1$. Then there exists a $\gamma_2\in\, ]0,r_1]$ such that 
$g(\gamma_2)=1$. It is easy to check that $A(\gamma_2\Id-B)^{-1}A^{-1}(a)=(\gamma_2\Id-T)^{-1}(a)$ is a fixed point of $\ol{T}_a$. \qed

\smallskip
\begin{remark} 
Note that for $T=-\Id$, a rotation by $r=(-1,0)$ on $\R^2$, $\overline{T}_a$ on $\mbb{S}^1$ has only four periodic points of order 2; namely 
$\ol{a},-\ol{a}, x_0, a-x_0$, where $x_0$ is such that $\|x_0\|=\|a-x_0\|=1$. It can also be shown that there exists a neighbourhood $U$ of $(-1,0)$ in $\mbb{S}^1$ 
such that for every $r\in U$, and 
$T$ a rotation by $r$ in $\R^2$, $\overline{T}_a$ on $\mbb{S}^1$ has four periodic points of order 2. The proof can be given by choosing a sufficiently small neighbourhood $U$ 
such that $\overline{T}^2_a$ either contracts or expands each of the four quadrants of $\mbb{S}^1$ (defined by the line passing through 0 and $a$, and its perpendicular), 
which would imply the existence of fixed points for 
$\overline{T}_a^2$ in each quadrant.
\end{remark}

So far we have discussed the existence of fixed points or periodic points of order 2 for $\T_a$ for a set of nonzero $a$ satisfying $\|T^{-1}(a)\|<1$, by putting conditions on the 
 eigenvalues of $T$. We now want to explore conditions under which there exists at least one nonzero $a$ such that $\|T^{-1}(a)\|<1$ and  $\T_a$ has a fixed point or a 
 periodic point of order 2.

\begin{corollary} \label{propx}
Let $T\in GL(2,\R)$. Suppose $T$ has either real eigenvalues or complex eigenvalues of the form $t(\cos\theta\pm i\sin\theta)$, ($t>0$), where either $0<\cos\theta<1$ or 
$\|T\|>5\sqrt{\det T}$. Then there exists an $a\in\R^2$ such that $0<\|T^{-1}(a)\|<1$ and $\ol{T}_a$ has a fixed point or a periodic point of order 2 and it is not distal. 
\end{corollary}

\begin{remark} Given any $T$ with complex eigenvalues, 
we can take $T'=CTC^{-1}$, which has the same eigenvalues as $T$ but the norm of $T'$ is very large. For $T=tABA^{-1}$,  $t^2=\det T$ as above with 
$B$ a rotation by an angle $\theta$, take $C=C(\beta)A^{-1}$, where 
$C(\beta)= \begin{bmatrix}
\beta & 0\\
0 & \beta^{-1}
\end{bmatrix}$. Then for $\beta>1$,  $T'$ has norm greater than $\beta^2|t\sin\theta|>5\sqrt{\det T}$ if $\beta>\sqrt{5}/|\sin\theta|$; (here, $\sin\theta\ne 0$).  
That is, given any $T\in GL(2,\R)$, there exist a conjugate $S$ of $T$ in $GL(2,\R)$ and a nonzero $a\in\R^2$ such that $\|S^{-1}(a)\|<1$ and $\ol{S}_a$ is not distal. 
\end{remark}

\smallskip
\noindent Proof of Corollary~\ref{propx}:
We may assume that $\det T= \pm 1$. If $T$ has at least one real positive eigenvalue then by Theorem~\ref{m}~(1), $\overline{T}_a$ has a fixed point for all $a$ satisfying  
$0<\|T^{-1}(a)\|<1$. If both 
the eigenvalues of $T$ are real and negative, then for an eigenvector $a$, $\ol{a}=a/\|a\|$ is a periodic point of order 2 for $\ol{T}_a$. 
Now suppose $T$ has complex eigenvalues $\cos\theta\pm i\sin\theta$ with $0<\cos\theta<1$. Then $\det T=1$. Suppose $T$ is an isometry. 
Since $|\sin\theta|\ne 1$, we can choose $a\in \R^2$ such that $|\sin\theta|<\|a\|<1$ and we have $\|T^{-1}(a)\|=\|a\|<1$. Now by Theorem~\ref{m}~(2), $\ol{T}_a$ has a 
fixed point. 
If $T$ is not an isometry, $T=ABA^{-1}$ and $\|T\|>1$. As $0<r_1=\cos\theta<1$, it is easy to check that $T^2$ is not an isometry and $\|T^2\|>1$. 
There exists $a\ne 0$ such that $\|T^{-1}(a)\|<1$ and $\|T(a)\|=\|T^2(T^{-1}(a))\|>1$. Now let $g:\R^+\to\R^+$ be defined as in the proof of Theorem~\ref{m}. Then 
$0<g(0)<1$ and $g(2r_1)=\|T(a)\|>1$. Therefore, there exists $\gamma_3\in\, ]0,2r_1[$, such that $g(\gamma_3)=1$ and hence 
$A(\gamma_3\Id-B)^{-1}A^{-1}(a)$ is a fixed point for $\ol{T}_a$. 
Now suppose $r_1\leq 0$. For $g$ as above, $g(1)=(1/[2(1-r_1)])\|a-T^{-1}(a)\|$. Since $\|T\|>5$, there exists $a$ such that $\|T^{-1}(a)\|<1$ and $\|a\|>5$.
Then $g(1)>(1/4)(5-1)=1$. Now there exists $\gamma_4$ such that $0<\gamma_4<1$ and $g(\gamma_4)=1$, and hence 
$A(\gamma_4\Id-B)^{-1}A^{-1}(a)=(\gamma_4\Id-T)^{-1}(a)$ is a 
fixed point for $\ol{T}_a$. 

As observed earlier, $\ol{T}_a^2$ is nontrivial. As $\T_a^2$ has a fixed point, this implies that $\T_a^2$, and hence, $\T_a$ is not distal.  
\qed

\smallskip
It is evident from Corollary~\ref{propx} that for a certain class of $T\in GL(2,\R)$ which have complex eigenvalues, we have not been able to show the existence of a nonzero 
$a$ such that $\T_a$ on $\mbb{S}^1$ has a fixed or periodic point or $\T_a$ is distal (or not), e.g.\ isometries with complex eigenvalues $\cos\theta+i\sin\theta$ for
which $\cos\theta\leq 0$, except for a small neighbourhood of $-\Id$ as mentioned in Remark~9 above. However, the following corollary shows that for any isometry $T$ in 
$GL(n+1,\R)$, $n$ even, there exists a nonzero $a$ such that $\|T^{-1}(a)\|<1$ and $\T_a$ is not distal. 
Unlike on $S^1$, there are distal homeomorphisms on higher dimensional spheres which have fixed points. 

\begin{corollary} \label{i} Let $n\in\N$ be even. Let $T$ be an isometry in $GL(n+1,\R)$. Then there exists an $a\in\R^{n+1}$ such that $0<\|T^{-1}(a)\|<1$ and $\ol{T}_{a}$ on $\mbb{S}^n$ 
is not distal. 
\end{corollary}

\begin{proof} Since $n$ is even, $T$ has at least one real eigenvalue.  Either $T$ has more than one real eigenvalues or all except one eigenvalues of $T$ are complex. 
In the first case, $T$ keeps a 2-dimensional subspace $V$ invariant and the restriction of $T$ to $V$ satisfies the condition in 
Corollary~\ref{propx} and hence there exists $a\in V$ for which the assertion holds.

Now suppose all except one eigenvalues of $T$ are complex. Then $T$ keeps a 3-dimensional subspace $W$ invariant such that $T|_W$ is an isometry. We may restrict 
$T$ to $W$ and assume that $n=2$. Let $a$ be such that $0<\|a\|<1$, $T(a)=\pm a$. Let $W'$ be a two dimensional subspace invariant 
under $T$ and the restriction of $T$ to $W'$ has complex eigenvalues, i.e.\ it is a nontrivial isometry.  
Let $U\in SL(3,\R)$ be such that $U(a)=a$ and $U(w)=T(w)$ for all $w\in W'$. Let $D\in GL(3,\R)$ be such that $D(a)=\pm a=T(a)$ and $D|_{W'}$ is the identity map. Then 
$U$ and $D$ both are isometries, $T=UD=DU$, $\ol{T}_a=U\ol{D}_a=\ol{D}_aU$, and hence $\T_a^m=U^m\ol{D}_a^m$ for all $m\in\N$. Let $b\in W'$ be a (nonzero) 
eigenvector of $D$ such that $D(b)=b$ and  $\langle a,b\rangle=0$. Let $V'$ be the subspace generated by $a$ and $b$. Then $D(V')=V'$, $D$ and $\ol{D}_a$ keep 
$V'\cap\mbb{S}^2$ invariant; where $V'\cap\mbb{S}^2$ is isomorphic to $\mbb{S}^1$ as we take the same norm on $V'$.  As $D|_{V'}$ has an eigenvalue 1, by 
Theorem~\ref{m}~(1), $\ol{D}_a$ has a fixed point in $V'\cap \mbb{S}^2$. As observed earlier, $\ol{D}_a^2$ restricted to $V'\cap \mbb{S}^2$ is nontrivial, and
there exist $x,y\in V'\cap \mbb{S}^2$, $x\ne y$, such that $\ol{D}_a^{2m}(x)\to z$ and $\ol{D}^{2m}(y)\to z$ as $m\to\infty$. Since $U$ is an 
isometry, there exists a sequence $\{m_k\}$ such that $U^{m_k}\to \Id$. Hence, $\ol{T}_a^{2m_k}(x)\to z$ and $\ol{T}_a^{2m_k}(y)\to z$ as $k\to\infty$.  
This implies that $\ol{T}_a$ is not diatal. 
\end{proof}

It would be interesting to extensively analyse the dynamics of `affine' maps $\T_a$ on higher dimensional spheres. 

\medskip
\noindent{\bf Acknowledgement}.
 A.\ K.\ Yadav is supported by a UGC-BSR fellowship.

{\scriptsize
{\noindent Riddhi Shah \hspace{178pt}{\noindent Alok Kumar Yadav}\\  {School of Physical Sciences}\hspace{127pt}{School of Physical Sciences}\\ {Jawaharlal Nehru University (JNU)}\hfill{Jawaharlal Nehru University (JNU)}\\ {New Delhi 110067, India}\hspace{136pt}{New Delhi 110067, India}\\ {\color{blue} rshah@jnu.ac.in}\hspace{169pt}{\color{blue} alokjnu90@gmail.com}\\ {\color{blue} riddhi.kausti@gmail.com}}

\end{document}